\documentclass[a4paper]{article}

\usepackage{amsthm}
\usepackage{amsmath}
\usepackage{bm}
\usepackage{booktabs}
\usepackage{color}
\usepackage{enumitem}
\usepackage[hang,flushmargin]{footmisc}
\usepackage[left=25mm,right=25mm,top=25mm,bottom=25mm]{geometry}
\usepackage{graphicx}
\usepackage{mathpazo}
\usepackage{mathtools}
\usepackage{microtype}
\usepackage{multicol}
\usepackage{sectsty}
\usepackage{setspace}
\usepackage{titlesec}
	\titlespacing{\section}{0pt}{12pt}{0pt}
	\titlespacing{\subsection}{0pt}{6pt}{0pt}
\usepackage[nottoc]{tocbibind}

\usepackage[
	pdftitle={Topological recursion and a	quantum curve for monotone Hurwitz numbers},
	pdfauthor={Norman Do, Alastair Dyer, and Daniel V. Mathews},
	ocgcolorlinks,
	linkcolor=linkblue,
	citecolor=linkred,
	urlcolor=linkred]
{hyperref}

\theoremstyle{plain}
\newtheorem{theorem}{Theorem}
\newtheorem{proposition}[theorem]{Proposition}
\newtheorem{corollary}[theorem]{Corollary}
\newtheorem{lemma}[theorem]{Lemma}

\newtheorem*{conjecture*}{Conjecture}

\theoremstyle{definition}
\newtheorem{definition}[theorem]{Definition}
\newtheorem{example}[theorem]{Example}

\theoremstyle{remark}
\newtheorem{remark}[theorem]{Remark}

\definecolor{linkred}{rgb}{0.75,0,0}
\definecolor{linkblue}{rgb}{0,0,1}

\newcommand\blfootnote[1]{
  \begingroup
  \renewcommand\thefootnote{}\footnote{#1}
  \addtocounter{footnote}{-1}
  \endgroup
}

\sectionfont{\large}
\subsectionfont{\normalsize}

\setlist{nolistsep}

\newcommand{\dd}{\mathrm{d}}
\newcommand{\h}{\hbar}
\newcommand{\z}{\overline{z}}
\renewcommand{\H}{\vec{H}}

\renewcommand{\P}{\vec{P}}
\newcommand{\mmu}{\boldsymbol{\mu}}
\DeclareRobustCommand{\stirling}{\genfrac\{\}{0pt}{}}

\makeatletter
\newcommand{\dashrule}[1][black]{
  \color{#1}\rule[\dimexpr.5ex-.2pt]{4pt}{.4pt}\xleaders\hbox{\rule{4pt}{0pt}\rule[\dimexpr.5ex-.2pt]{4pt}{.4pt}}\hfill\kern0pt
}
\newcommand{\rulecolor}[1]{
  \def\CT@arc@{\color{#1}}
}
\makeatother

\begin{document}

{\large \bfseries Topological recursion and a quantum curve for monotone Hurwitz numbers}

{\bfseries Norman Do, Alastair Dyer, and Daniel V. Mathews}

\emph{Abstract.} Classical Hurwitz numbers count branched covers of the Riemann sphere with prescribed ramification data, or equivalently, factorisations in the symmetric group with prescribed cycle structure data. Monotone Hurwitz numbers restrict the enumeration by imposing a further monotonicity condition on such factorisations. In this paper, we prove that monotone Hurwitz numbers arise from the topological recursion of Eynard and Orantin applied to a particular spectral curve. We furthermore derive a quantum curve for monotone Hurwitz numbers. These results extend the collection of enumerative problems known to be governed by the paradigm of topological recursion and quantum curves, as well as the list of analogues between monotone Hurwitz numbers and their classical counterparts.
\blfootnote{
\emph{Date:} 18 August 2014 \\ The first author was supported by the Australian Research Council grant DE130100650.}

~

\hrule

\setlength{\parskip}{1pt}
\tableofcontents
\setlength{\parskip}{6pt}

\section{Introduction} \label{sec:intro}

The simple Hurwitz number $H_{g,n}(\mmu)$ counts genus $g$ branched covers of $\mathbb{CP}^1$ with simple ramification, except for the ramification profile $\mmu = (\mu_1, \mu_2, \ldots, \mu_n)$ over $\infty$. The monodromy of such a branched cover associates a transposition in the symmetric group $S_{\lvert \mmu \rvert}$ to each simple branch point, where $\lvert \mmu \rvert = \mu_1 + \mu_2 + \cdots + \mu_n$. The product of these transpositions is necessarily the inverse of the monodromy permutation associated to $\infty$, which has cycle type $\mmu$. Conversely, the Riemann existence theorem guarantees that any such factorisation in the symmetric group corresponds to a branched cover of $\mathbb{CP}^1$ with the desired ramification data.

The monotone Hurwitz number $\H_{g,n}(\mmu)$ counts those factorisations in the symmetric group enumerated by $H_{g,n}(\mmu)$ that satisfy an additional constraint. We require the transpositions $(a_1 ~ b_1), (a_2 ~ b_2), \ldots, (a_m ~ b_m)$ of the factorisation to satisfy the property of monotonicity --- namely, that when the transpositions are written with the convention that $a_i < b_i$, then $b_1 \leq b_2 \leq \cdots \leq b_m$. The monotone Hurwitz numbers first appeared in a series of papers by Goulden, Guay--Paquet and Novak, in which they arose as coefficients in the large $N$ asymptotic expansion of the Harish-Chandra--Itzykson--Zuber (HCIZ) matrix integral over the unitary group $U(N)$~\cite{gou-gua-nov13a, gou-gua-nov13b, gou-gua-nov14a}. The monotonicity condition is also natural from the standpoint of the Jucys--Murphy elements in the symmetric group algebra $\mathbb{C}[S_{|\mmu|}]$.

In this paper, we prove that the monotone Hurwitz numbers are governed by the topological recursion of Eynard and Orantin. The topological recursion was inspired by the loop equations from the theory of matrix models~\cite{eyn-ora07a}. The topological recursion takes as input the data of a spectral curve, which we take to be a Torelli marked compact Riemann surface $\mathcal C$ endowed with two meromorphic functions $x$ and $y$. The output is an infinite family of meromorphic multidifferentials $\omega_{g,n}$ on ${\mathcal C}^n$ for integers $g \geq 0$ and $n \geq 1$, which we refer to as correlation differentials.

The coefficients of certain series expansions of correlation differentials are often solutions to problems in enumerative geometry and mathematical physics. In this way, topological recursion governs intersection theory on moduli spaces of curves~\cite{eyn-ora07a}, Weil--Petersson volumes of moduli spaces of hyperbolic surfaces~\cite{eyn-ora07b}, enumeration of ribbon graphs~\cite{nor13, dum-mul-saf-sor}, stationary Gromov--Witten theory of $\mathbb{P}^1$~\cite{nor-sco, dun-ora-sha-spi}, simple Hurwitz numbers and their generalisations~\cite{bou-mar, eyn-mul-saf, do-lei-nor, bou-her-liu-mul}, and Gromov--Witten theory of toric Calabi--Yau threefolds~\cite{bou-kle-mar-pas, eyn-ora12, fan-liu-zon}. There are also conjectural relations to spin Hurwitz numbers~\cite{mul-sha-spi} and quantum invariants of knots~\cite{dij-fuj-man, bor-eyn}. These myriad applications raise questions such as the following: How universal is the scope of topological recursion? What is the commonality among problems governed by topological recursion?

In order to state our result relating monotone Hurwitz numbers to topological recursion, define the following generating functions, known as \emph{free energies}.
\[
F_{g,n}(z_1, z_2, \ldots, z_n) = \sum_{\mu_1, \mu_2, \ldots, \mu_n = 1}^\infty \H_{g,n}(\mu_1, \mu_2, \ldots, \mu_n)~x_1^{\mu_1} x_2^{\mu_2} \cdots x_n^{\mu_n}
\]
Here and throughout the paper, we use the shorthand $x_i$ to denote $x(z_i)$. We will show in Corollary~\ref{cor:fenergies} that the free energy $F_{g,n}$ is a meromorphic function on the $n$-fold Cartesian product ${\mathcal C}^n$ of the spectral curve. This fact is used to prove the following result, which had been previously suggested by Goulden, Guay--Paquet and Novak~\cite{gou-gua-nov13a}.

\begin{theorem} \label{thm:toprec}
The topological recursion applied to the spectral curve $\mathbb{CP}^1$ with functions $x, y: \mathbb{CP}^1 \to \mathbb{C}$ given by
\[
x(z) = \frac{z-1}{z^2} \qquad \text{and} \qquad y(z) = -z
\]
produces correlation differentials whose expansions at $x_i = 0$ satisfy
\[
\omega_{g,n}(z_1, z_2, \ldots, z_n) = \frac{\partial}{\partial x_1} \frac{\partial}{\partial x_2} \cdots \frac{\partial}{\partial x_n} F_{g,n}(z_1, z_2, \ldots, z_n)~\dd x_1 \otimes \dd x_2 \otimes \cdots \otimes \dd x_n, \qquad \text{for $(g,n) \neq (0,2)$}.
\]
\end{theorem}

It was posited by Gukov and Su{\l}kowski that spectral curves $A(x, y) = 0$ satisfying a certain K-theoretic criterion may be quantised to produce a non-commutative \emph{quantum curve} $\widehat{A}(\widehat{x}, \widehat{y})$. One can interpret $\widehat{A}(\widehat{x}, \widehat{y})$ as a differential operator via $\widehat{x} = x$ and $\widehat{y} = -\h \frac{\partial}{\partial x}$, and it is natural to consider the following Schr\"{o}dinger-like equation~\cite{guk-sul}.
\[
\widehat{A}(\widehat{x}, \widehat{y}) \, Z(x, \h) = 0
\]
Gukov and Su{\l}kowski conjecture that the solution $Z(x, \h)$ possesses a perturbative expansion that can be calculated from the spectral curve via the topological recursion. Conversely, they suggest that the Schr\"{o}dinger-like equation may be used to recover the quantisation $\widehat{A}(\widehat{x}, \widehat{y})$. Quantum curves have been rigorously shown to exist in the sense of Gukov and Su{\l}kowski for several problems, including intersection theory on moduli spaces of curves~\cite{zho12a}, enumeration of ribbon graphs~\cite{mul-sul}, simple Hurwitz numbers and their generalisations~\cite{mul-sha-spi}, and open string invariants for $\mathbb{C}^3$ and the resolved conifold~\cite{zho12b}.

The \emph{wave function} $Z(x, \h)$ --- also referred to as the \emph{partition function} --- assembles the free energies into a generating function in the following way.
\begin{equation} \label{eq:pfunction}
Z(x, \hbar) = \exp \left[ \sum_{g=0}^\infty \sum_{n=1}^\infty \frac{\hbar^{2g-2+n}}{n!} ~ F_{g,n}(z, z, \ldots, z) \right]
\end{equation}

\begin{theorem} \label{thm:qcurve}
The wave function $Z(x, \hbar)$ satisfies the following equation, where $\widehat{x} = x$ and $\widehat{y} = -\hbar \frac{\partial}{\partial x}$.
\[
[\widehat{x} \widehat{y}^2 + \widehat{y} + 1] \, Z(x, \hbar) = 0
\]
\end{theorem}

Theorem~\ref{thm:toprec} implies that the spectral curve for monotone Hurwitz numbers is given by the equation $A(x, y) = xy^2 + y + 1$, while Theorem~\ref{thm:qcurve} implies that the quantum curve is given by the equation $\widehat{A}(\widehat{x}, \widehat{y}) = \widehat{x} \widehat{y}^2 + \widehat{y} + 1$. These results extend the collection of enumerative problems known to be governed by the paradigm of topological recursion and quantum curves to include monotone Hurwitz numbers. They also provide a new analogue between monotone Hurwitz numbers and their classical counterparts. It is possible to obtain further results by applying the general theory of topological recursion. For example, the asymptotic behaviour of monotone Hurwitz numbers stores intersection numbers on the moduli space $\overline{\mathcal M}_{g,n}$ of stable pointed curves, while the string and dilaton equations imply additional properties of monotone Hurwitz numbers --- see ~Propositions~\ref{pro:coeffs} and~\ref{pro:dilaton}, respectively.

The structure of the paper is as follows.
\begin{itemize}
\item In Section~\ref{sec:monohurwitz}, we state the definition of the monotone Hurwitz numbers, the monotone cut-and-join recursion, and a polynomiality result, all of which appear in the work of Goulden, Guay--Paquet and Novak~\cite{gou-gua-nov13a, gou-gua-nov13b, gou-gua-nov14a}. We furthermore derive some useful consequences.
\item In Section~\ref{sec:toprec}, we briefly review the topological recursion of Eynard and Orantin as well as the related notion of a quantum curve.
\item In Section~\ref{sec:theorem1}, we prove Theorem~\ref{thm:toprec}, which relates the monotone Hurwitz numbers to the topological recursion of Eynard and Orantin applied to a particular rational spectral curve. The result is deduced from the cut-and-join recursion and polynomiality for monotone Hurwitz numbers.
\item In Section~\ref{sec:theorem2}, we explicitly calculate the wave function for the monotone Hurwitz numbers and prove Theorem~\ref{thm:qcurve}.
\item In Section~\ref{sec:applications}, we apply the general theory of topological recursion to relate monotone Hurwitz numbers to intersection numbers on $\overline{\mathcal M}_{g,n}$ as well as to derive string and dilaton equations.
\end{itemize}

\section{Monotone Hurwitz numbers} \label{sec:monohurwitz}

\subsection{Definition and motivation}

Classical Hurwitz theory studies the enumeration of branched covers of Riemann surfaces with prescribed ramification data. As an example, the simple Hurwitz numbers enumerate branched covers of $\mathbb{CP}^1$ with simple ramification except for over $\infty$. More precisely, we have the following definition.

\begin{definition}
The \emph{simple Hurwitz number} $H_{g,n}(\mu_1, \mu_2, \ldots, \mu_n)$ is the weighted count,
up to topological equivalence, of connected genus $g$ branched covers $f: C \to \mathbb{CP}^1$ such that
\begin{itemize}
\item the preimages of $\infty$ are labelled $p_1, p_2, \ldots, p_n$ and the divisor $f^{-1}(\infty)$ is equal to $\mu_1 p_1 + \mu_2 p_2 + \cdots + \mu_n p_n$; and
\item the only other ramification is simple and occurs over $m = 2g - 2 + n + |\mmu|$ fixed points.
\end{itemize}

We consider two branched covers $f_1: C_1 \to \mathbb{CP}^1$ and $f_2: C_2 \to \mathbb{CP}^1$ \emph{topologically equivalent}  if there exists a homeomorphism $\phi: C_1 \to C_2$ such that $f_1 = f_2 \circ \phi$. We count a branched cover $f$ with the weight $\frac{1}{\# \mathop{\mathrm{Aut}} f}$, where $\mathop{\mathrm{Aut}} f$ denotes the group of automorphisms of $f$.
\end{definition}

Throughout this paper, we use the notation $|\mmu| = \mu_1 + \mu_2 + \cdots + \mu_n$ for an $n$-tuple of positive integers $\mmu = (\mu_1, \mu_2, \ldots, \mu_n)$. The equality $m = 2g - 2 + n + |\mmu|$ is a direct consequence of the Riemann--Hurwitz formula.

By considering the monodromy of such a branched cover, one may associate to each branch point with ramification profile $\lambda$ a permutation in $S_{|\mmu|}$ with cycle type $\lambda$. Furthermore, the monodromy permutations away from $\infty$ necessarily multiply 
(when taken in the correct order)
to give the inverse of the monodromy permutation over $\infty$. Conversely, the Riemann existence theorem guarantees that there exists a unique branched cover with the corresponding ramification data, once the location of each branch point is fixed. Therefore, the definition of the simple Hurwitz numbers may be rephrased in terms of factorisations in the symmetric group in the following way.

\begin{proposition}
The simple Hurwitz number $H_{g,n}(\mu_1, \mu_2, \ldots, \mu_n)$ is equal to $\frac{1}{|\mmu|!}$ multiplied by the number of $m$-tuples $(\sigma_1, \sigma_2, \ldots, \sigma_m)$ of transpositions in the symmetric group $S_{|\mmu|}$ such that
\begin{itemize}
\item $m = 2g - 2 + n + |\mmu|$; 
\item the cycles of $\sigma_1 \circ \sigma_2 \circ \cdots \circ \sigma_m$ are labelled $1, 2, \ldots, n$ such that cycle $i$ has length $\mu_i$ for $i = 1, 2, \ldots, n$; and
\item  $\sigma_1, \sigma_2, \ldots, \sigma_m$ generate a transitive subgroup of $S_{|\mmu|}$.
\end{itemize}
\end{proposition}

The monotone Hurwitz numbers were introduced in a series of papers by Goulden, Guay--Paquet and Novak~\cite{gou-gua-nov13a, gou-gua-nov13b, gou-gua-nov14a}. They are obtained by adding a monotonicity condition to the definition of the simple Hurwitz numbers.

\begin{definition} \label{def:mhurwitz}
The \emph{monotone Hurwitz number} $\vec{H}_{g,n}(\mu_1, \mu_2, \ldots, \mu_n)$ is equal to $\frac{1}{|\mmu|!}$ multiplied by the number of $m$-tuples $(\sigma_1, \sigma_2, \ldots, \sigma_m)$ of transpositions in the symmetric group $S_{|\mmu|}$ such that
\begin{itemize}
\item $m = 2g - 2 + n + |\mmu|$;
\item the cycles of $\sigma_1 \circ \sigma_2 \circ \cdots \circ \sigma_m$ are labelled $1, 2, \ldots, n$ such that cycle $i$ has length $\mu_i$ for $i = 1, 2, \ldots, n$;
\item  $\sigma_1, \sigma_2, \ldots, \sigma_m$ generate a transitive subgroup of $S_{|\mmu|}$; and
\item if we write each transposition $\sigma_i = (a_i ~ b_i)$ with $a_i < b_i$, then $b_1 \leq b_2 \leq \cdots \leq b_m$.
\end{itemize}
\end{definition}

\begin{remark}
Note that the simple Hurwitz numbers and monotone Hurwitz numbers defined here may differ from those appearing elsewhere in the literature by simple combinatorial factors. For example, the monotone Hurwitz numbers defined by Goulden, Guay--Paquet and Novak~\cite{gou-gua-nov13a, gou-gua-nov13b, gou-gua-nov14a} are related to ours in the following way, where $\text{Aut } \mmu$ denotes the set of permutations of the tuple $\mmu = (\mu_1, \mu_2, \ldots, \mu_n)$ that leave it fixed.
\[
\H_{g,n}^{\text{GGN}}(\mmu) = \H_{g,n}^{\text{DDM}}(\mmu) \times \frac{|\mmu|!}{\# \text{Aut } \mmu}
\]
Our normalisation of the monotone Hurwitz numbers is particularly well-suited to our purposes.
\end{remark}

\begin{example}
Of the 27 factorisations in $S_3$ with three transpositions, there are 24 that satisfy the transitivity property.
\begin{center}
\begin{multicols}{4}
 $(1~2) \circ (1~2) \circ (1~3)$ \\
 $(1~2) \circ (1~2) \circ (2~3)$ \\
 $(1~2) \circ (1~3) \circ (1~3)$ \\
 $(1~2) \circ (1~3) \circ (2~3)$ \\
 $(1~2) \circ (2~3) \circ (1~3)$ \\
 $(1~2) \circ (2~3) \circ (2~3)$ \\
 $(1~3) \circ (1~3) \circ (2~3)$ \\
 $(1~3) \circ (2~3) \circ (1~3)$ \\
 $(1~3) \circ (2~3) \circ (2~3)$ \\
 $(2~3) \circ (1~3) \circ (1~3)$ \\
 $(2~3) \circ (1~3) \circ (2~3)$ \\
 $(2~3) \circ (2~3) \circ (1~3)$
\end{multicols}

\begin{multicols}{4}
 $(1~2) \circ (1~3) \circ (1~2)$ \\
 $(1~2) \circ (2~3) \circ (1~2)$ \\
 $(1~3) \circ (1~2) \circ (1~2)$ \\
 $(1~3) \circ (1~2) \circ (1~3)$ \\
 $(1~3) \circ (1~2) \circ (2~3)$ \\
 $(1~3) \circ (1~3) \circ (1~2)$ \\
 $(1~3) \circ (2~3) \circ (1~2)$ \\
 $(2~3) \circ (1~2) \circ (1~2)$ \\
 $(2~3) \circ (1~2) \circ (1~3)$ \\
 $(2~3) \circ (1~2) \circ (2~3)$ \\
 $(2~3) \circ (1~3) \circ (1~2)$ \\
 $(2~3) \circ (2~3) \circ (1~2)$
\end{multicols}
\end{center}
Since all such products result in a permutation of cycle type $(1, 2)$, we obtain $H_{0,2}(1,2) = \frac{24}{3!} = 4$. Of these 24 factorisations, only the first 12 are monotone, so we obtain $\H_{0,2}(1,2) = \frac{12}{3!} = 2$.
\end{example}

More generally, one can define \emph{double monotone Hurwitz numbers} $\H_{g, m, n} (\boldsymbol{\lambda}; \mmu)$, which enumerate branched covers of $\mathbb{CP}^1$ with simple ramification except for the ramification profile $\boldsymbol{\lambda} = (\lambda_1, \lambda_2, \ldots, \lambda_m)$ over 0 and the ramification profile $\mmu = (\mu_1, \mu_2, \ldots, \mu_n)$ over $\infty$.
The initial motivation for their study was the work of Goulden, Guay--Paquet and Novak on the asymptotic expansion of the HCIZ matrix integral over the unitary group~\cite{gou-gua-nov14a}.
\[
{\cal I}_N(z; A, B) = \int_{U(N)} \exp \left[ -zN ~ \mathrm{Tr}(AUBU^{-1}) \right] \dd U
\]
The monotone Hurwitz numbers arise naturally as coefficients in the large $N$ asymptotic expansion of $\frac{1}{N^2} \log {\cal I}_N(z; A, B)$.

It is worth remarking that the monotonicity condition for Hurwitz numbers is also natural from the standpoint of the Jucys--Murphy elements of the symmetric group algebra. These are defined as
\[
J_k = (1 ~ k) + (2 ~ k) + \cdots + (k-1 ~~ k) \in \mathbb{C}[S_d], \qquad \text{for } k = 2, 3, \ldots, d.
\]
A theorem of Jucys states that symmetric polynomials in the Jucys--Murphy elements generate the centre $Z\mathbb{C}[S_d]$ of the symmetric group algebra~\cite{juc}.

If we remove the transitivity property from Definition~\ref{def:mhurwitz}, we obtain the \emph{disconnected monotone Hurwitz numbers} $\H^\bullet_{g,n}(\mmu)$. One can check that that these are related to the Jucys--Murphy elements via the equation
\[
\prod_{i=1}^n \mu_i \times \H^\bullet_{g,n}(\mmu) = [C_{\mmu}] \, h_m(J_2, J_3, J_4, \ldots, J_{|\mmu|}).
\]
In this formula, $m = 2g - 2 + n + |\mmu|$ as in Definition~\ref{def:mhurwitz} and $h_m$ denotes the complete homogeneous symmetric polynomial of degree $m$. The notation $[C_{\mmu}]$ on the right hand side requires us to take the coefficient of the conjugacy class $C_{\mmu}$ in the subsequent expression. One can relate the connected monotone Hurwitz numbers to the disconnected ones by forming the appropriate exponential generating functions and applying the exponential formula. This technique is applied in the proof of Proposition~\ref{pro:zcoefficients}.

Beyond making explicit the connection between monotone Hurwitz numbers and the HCIZ integral, Goulden, Guay--Paquet and Novak deduce results for monotone Hurwitz numbers that have analogues in the classical theory of Hurwitz numbers. The two that are of particular importance to us are the cut-and-join recursion and polynomiality, which we present in the remainder of this section.

\subsection{Monotone cut-and-join recursion}

The classical cut-and-join recursion expresses a simple Hurwitz number in terms of Hurwitz numbers of lesser complexity. The measure of complexity of the Hurwitz number $H_{g,n}(\mmu)$ is $m = 2g - 2 + n + |\mmu|$, which is the number of transpositions in the corresponding factorisation, or equivalently, the number of simple branch points in the corresponding branched cover. So the cut-and-join recursion provides an effective recursion for computing simple Hurwitz numbers from the base case $H_{0,1}(1) = 1$. Goulden, Guay--Paquet and Novak prove the following version of the cut-and-join recursion in the monotone case.

\begin{proposition}[Monotone cut-and-join recursion~\cite{gou-gua-nov13a}] \label{pro:cutjoin}
The monotone Hurwitz numbers satisfy the following recursion, where $S = \{1, 2, \ldots, n\}$ and $\mmu_I = (\mu_{i_1}, \mu_{i_2}, \ldots, \mu_{i_k})$ for $I = \{i_1, i_2, \ldots, i_k\}$.
\begin{align*}
\mu_1 \H_{g,n}(\mmu_S) =& \sum_{i = 2}^n (\mu_1+\mu_i) \, \H_{g,n-1}(\mu_1+\mu_i, \mmu_{S \setminus \{1,i\}}) + \sum_{\alpha + \beta = \mu_1} \alpha \beta \, \H_{g-1,n+1}(\alpha, \beta, \mmu_{S \setminus \{1\}}) \\
&+ \sum_{\alpha + \beta = \mu_1} \mathop{\sum_{g_1+g_2=g}}_{I \sqcup J = S \setminus \{1\}} \alpha \beta \, \H_{g_1, |I|+1}(\alpha, \mmu_I) \, \H_{g_2, |J|+1}(\beta, \mmu_J)
\end{align*}
\end{proposition}

The monotone cut-and-join recursion can be proved in a similar way to its classical counterpart. One considers the effect of multiplying both sides of the equation
\[
(a_1 ~ b_1) \circ (a_2 ~ b_2) \circ \cdots \circ (a_m ~ b_m) = \sigma
\]
on the right by the transposition $(a_m ~ b_m)$. If $a_m$ and $b_m$ belong to the same cycle of $\sigma$, then that cycle is cut into two in the product $\sigma \circ (a_m ~ b_m)$. Otherwise, $a_m$ and $b_m$ belong to two disjoint cycles of $\sigma$ and these are joined into one in the product $\sigma \circ (a_m ~ b_m)$. A careful combinatorial analysis of these cases produces the desired result.

We recall from Section~\ref{sec:intro} that the free energies for the monotone Hurwitz problem are the following natural generating functions.
\[
F_{g,n}(z_1, z_2, \ldots, z_n) = \sum_{\mu_1, \mu_2, \ldots, \mu_n = 1}^\infty \H_{g,n}(\mu_1, \mu_2, \ldots, \mu_n)~x_1^{\mu_1} x_2^{\mu_2} \cdots x_n^{\mu_n}
\]

\begin{lemma} \label{lem:fenergy01}
The monotone Hurwitz numbers of type $(0,1)$ are given by the formula $\H_{0,1}(\mu) = \binom{2\mu}{\mu} \, \frac{1}{2\mu (2\mu-1)}$. The corresponding free energy satisfies the equation $\frac{\partial}{\partial x_1} F_{0,1}(z_1) = z_1$.
\end{lemma}

\begin{proof}
The monotone cut-and-join recursion in the case $(g,n) = (0,1)$ reads
\[
\mu \H_{0,1}(\mu) = \sum_{\alpha + \beta = \mu} \alpha \beta \, \H_{0,1}(\alpha) \, \H_{0,1}(\beta).
\]
Writing $\mu \H_{0,1}(\mu) = C_{\mu-1}$, we see that this is precisely the recursion for the Catalan numbers. Given that the base case $1 \H_{0,1}(1) = C_0 = 1$ is in agreement, we obtain $\H_{0,1}(\mu) = \frac{1}{\mu} C_{\mu-1} = \binom{2\mu}{\mu} \, \frac{1}{2\mu (2\mu-1)}$, as required. 

One then obtains the relation for the corresponding free energy via the following chain of equalities.
\[
\frac{\partial}{\partial x_1} F_{0,1}(z_1) = \sum_{\mu=1}^\infty \mu \H_{0,1}(\mu) x_1^{\mu-1} = \sum_{\mu=0}^\infty C_\mu x_1^{\mu} = \frac{1 - \sqrt{1-4x_1}}{2x_1} = z_1
\]
The last two equalities follow from the well-known generating function for Catalan numbers and the spectral curve relation $x = \frac{z-1}{z^2}$.
\end{proof}

\begin{proposition} \label{pro:cutjoingf}
For $(g,n) \neq (0,1)$, the free energies satisfy the following recursion, where the $\circ$ over the final summation means that we exclude all terms involving $F_{0,1}$.
\begin{align*}
(1-2x_1 z_1) x_1 \frac{\partial}{\partial x_1} F_{g,n}(z_S) =&~ \sum_{i=2}^n \frac{x_1x_i}{x_1-x_i} \left[ \frac{\partial}{\partial x_1} F_{g,n-1}(z_{S \setminus \{i\}}) - \frac{\partial}{\partial x_i} F_{g,n-1}(z_{S \setminus \{1\}}) \right] \\
&+ \left[ x(t_1) x(t_2) \, \frac{\partial^2}{\partial x(t_1) \, \partial x(t_2)} F_{g-1,n+1}(t_1, t_2, z_{S \setminus \{1\}}) \right]_{t_1 = t_2 = z_1} \\
&+ \mathop{\sum_{g_1 + g_2 = g}}_{I \sqcup J = S \setminus \{1\}}^\circ \left[ x_1 \frac{\partial}{\partial x_1} F_{g_1, |I|+1}(z_1, z_I) \right] \left[ x_1 \frac{\partial}{\partial x_1} F_{g_2, |J|+1}(z_1, z_J) \right]
\end{align*}
\end{proposition}

\begin{proof}
Multiply both sides of the monotone cut-and-join recursion of Proposition~\ref{pro:cutjoin} by the monomial $x_1^{\mu_1} x_2^{\mu_2} \cdots x_n^{\mu_n}$ and sum over all positive integers $\mu_1, \mu_2, \ldots, \mu_n$. There are three types of terms that must be dealt with.
\begin{itemize}
\item The left hand side yields terms of the following form.
\[
\sum_{\mu_1=1}^\infty \mu_1 \H(\mu_1) \, x_1^{\mu_1} = x_1 \frac{\partial}{\partial x_1} F(z_1)
\]
\item The first term on the right hand side yields terms of the following form.
\begin{align*}
\sum_{\mu_1,\mu_i=1}^\infty (\mu_1 + \mu_i) \, \H(\mu_1 + \mu_i) \, x_1^{\mu_1} x_i^{\mu_i} &= \sum_{\mu=1}^\infty \mu \H(\mu) ( x_1^{\mu-1} x_i + x_1^{\mu-2} x_i^2 + \cdots + x_1 x_i^{\mu-1} ) \\
&= \sum_{\mu=1}^\infty \mu \H(\mu) \, \frac{x_1 x_i}{x_1 - x_i} (x_1^{\mu-1} - x_i^{\mu-1}) \\
&= \frac{x_1 x_i}{x_1 - x_i} \left[ \frac{\partial}{\partial x_1} F(z_1) - \frac{\partial}{\partial x_i} F(z_i) \right]
\end{align*}
\item The last two terms on the right hand side yield terms of the following form.
\begin{align*}
\sum_{\mu_1=1}^\infty \sum_{\alpha + \beta = \mu_1} \alpha \beta \, \H(\alpha, \beta) \, x_1^{\mu_1} &= \sum_{\alpha, \beta = 1}^\infty \alpha \beta \, \H(\alpha, \beta) \, x_1^\alpha x_1^\beta \\
&= \left[ x(t_1) x(t_2) \, \frac{\partial^2}{\partial x(t_1) \, \partial x(t_2)} F(t_1, t_2) \right]_{t_1 = t_2 = z_1}
\end{align*}
\end{itemize}
After applying these transformations to the terms in the monotone cut-and-join recursion, move all terms involving $F_{0,1}$ in the last line on the right hand side to the left hand side. The desired result then follows after substituting $\frac{\partial}{\partial x_1} F_{0,1}(z_1) = z_1$, which is a consequence of Lemma~\ref{lem:fenergy01}.
\end{proof}

\begin{proposition} \label{pro:fenergy02}
The monotone Hurwitz numbers of type $(0,2)$ are given by the formula $\H_{0,2}(\mu_1, \mu_2) = \binom{2\mu_1}{\mu_1} \binom{2\mu_2}{\mu_2} \, \frac{1}{2(\mu_1+\mu_2)}$.
\end{proposition}

\begin{proof}
Proposition~\ref{pro:cutjoingf} in the case $(g,n) = (0,2)$ reads
\[
(1-2x_1 z_1) x_1 \frac{\partial}{\partial x_1} F_{0,2}(z_1, z_2) = \frac{x_1x_2}{x_1-x_2} \left[ \frac{\partial}{\partial x_1} F_{0,1}(z_1) - \frac{\partial}{\partial x_2} F_{0,1}(z_2) \right]. \]
Now substitute $x = \frac{z-1}{z^2}$ and $\frac{\partial}{\partial x_1} F_{0,1}(z_1) = z_1$ to obtain
\[
x_1 \frac{\partial}{\partial x_1} F_{0,2}(z_1, z_2) = \frac{z_1 (z_1-1) (z_2-1)}{(z_1-2) (z_1z_2-z_1-z_2)}.
\]
By the symmetry in $F_{0,2}(z_1, z_2)$, we may write
\begin{align*}
\left( x_1 \frac{\partial}{\partial x_1} + x_2 \frac{\partial}{\partial x_2} \right) F_{0,2}(z_1, z_2) &= \frac{z_1 (z_1-1) (z_2-1)}{(z_1-2) (z_1z_2-z_1-z_2)} + \frac{z_2 (z_1-1) (z_2-1)}{(z_2-2) (z_1z_2-z_1-z_2)} \\
&= 2 \frac{(z_1-1)(z_2-1)}{(z_1-2)(z_2-2)}.
\end{align*}
Consider the fact that $-2\frac{z-1}{z-2} = \sum_{\mu=1}^\infty \binom{2\mu}{\mu} \, x^\mu$, which follows from applying $\frac{\partial}{\partial x} \left[ x \, \cdot \, \right]$ to the Catalan generating function used in the proof of Lemma~\ref{lem:fenergy01}. Substituting this into the previous equation yields
\[
\left( x_1 \frac{\partial}{\partial x_1} + x_2 \frac{\partial}{\partial x_2} \right) F_{0,2}(z_1, z_2) = \frac{1}{2} \sum_{\mu_1, \mu_2 = 1}^\infty \binom{2\mu_1}{\mu_1} \binom{2\mu_2}{\mu_2} x_1^{\mu_1} x_2^{\mu_2},
\]
from which it follows that
\[
F_{0,2}(z_1, z_2) = \sum_{\mu_1, \mu_2 = 1}^\infty \binom{2\mu_1}{\mu_1} \binom{2\mu_2}{\mu_2} \, \frac{1}{2(\mu_1+\mu_2)} ~ x_1^{\mu_1} x_2^{\mu_2}.
\]
The desired expression for $\H_{0,2}(\mu_1, \mu_2)$ is obtained by comparing the coefficient of $x_1^{\mu_1} x_2^{\mu_2}$ on both sides of this equation.
\end{proof}

\subsection{Polynomiality}

One may observe the appearance of the central binomial coefficients $\binom{2\mu}{\mu}$ in formulas for the monotone Hurwitz numbers $\H_{0,1}(\mu)$ and $\H_{0,2}(\mu_1, \mu_2)$ above. In fact, this behaviour persists for all monotone Hurwitz numbers and we have the following polynomiality result due to Goulden, Guay--Paquet and Novak.

\begin{proposition}[Polynomiality of monotone Hurwitz numbers~\cite{gou-gua-nov13b}] \label{pro:polynomiality}
The monotone Hurwitz numbers may be expressed as
\[
\H_{g,n}(\mu_1, \mu_2, \ldots, \mu_n) = \prod_{i=1}^n \binom{2\mu_i}{\mu_i} \times \P_{g,n}(\mu_1, \mu_2, \ldots, \mu_n),
\]
where $\P_{g,n}$ is a symmetric rational function. For $(g,n) \neq (0,1)$ or $(0,2)$, $\P_{g,n}$ is a polynomial with rational coefficients of degree $3g-3+n$.
\end{proposition}

The table below contains $\P_{g,n}$ for various small values of $g$ and $n$.
\rulecolor{linkblue}
\begin{center}
\begin{tabular}{ll} \toprule
$(g,n)$ & $\P_{g,n}(\mu_1, \mu_2, \ldots, \mu_n)$ \\ \midrule
$(0,1)$ & $\frac{1}{2} \left[ 2\mu_1^2 - \mu_1 \right]^{-1}$ \\
$(0,2)$ & $\frac{1}{2} (\mu_1+\mu_2)^{-1}$ \\
$(0,3)$ & $1$ \\
$(0,4)$ & $2\mu_1 + 2\mu_2 + 2\mu_3 + 2\mu_4 + 1$ \\
$(1,1)$ & $\frac{1}{12} (\mu_1 - 1)$ \\
$(1,2)$ & $\frac{1}{12} (2\mu_1^2 + 2\mu_2^2 + 2\mu_1\mu_2 - \mu_1 - \mu_2 - 1)$ \\
$(2,1)$ & $\frac{1}{720} (10\mu_1^4 - 7\mu_1^3 - 16\mu_1^2 + 7 \mu_1 + 6)$ \\ \bottomrule
\end{tabular}
\end{center}

The polynomiality of monotone Hurwitz numbers endows a certain structure on the free energies $F_{g,n}(z_1, z_2, \ldots, z_n)$. In order to state the result, write the polynomial $\P_{g,n}$ in Proposition~\ref{pro:polynomiality} as
\[
\P_{g,n}(\mu_1, \mu_2, \ldots, \mu_n) = \sum_{a_1, a_2, \ldots, a_n = 0}^\text{finite} C_{g,n}(a_1, a_2, \ldots, a_n) ~ \mu_1^{a_1} \mu_2^{a_2} \cdots \mu_n^{a_n}.
\]

Combining this notation with Proposition~\ref{pro:polynomiality} and the definition of the free energies allows us to express them as follows.
\begin{equation} \label{eq:fenergies}
F_{g,n}(z_1, z_2, \ldots, z_n) = \sum_{a_1, a_2, \ldots, a_n = 0}^\text{finite} C_{g,n}(a_1, a_2, \ldots, a_n) \, \prod_{i=1}^n \sum_{\mu_i=1}^\infty \binom{2\mu_i}{\mu_i} \mu_i^{a_i} x_i^{\mu_i}
\end{equation}

It is thus natural to define the auxiliary functions
\begin{equation} \label{eq:functions}
f_a(z) = \sum_{\mu=1}^\infty \binom{2\mu}{\mu} \mu^a x^\mu = \left( x \frac{\partial}{\partial x} \right)^a \sum_{\mu=1}^\infty \binom{2\mu}{\mu} x^\mu = \left( -\frac{z(z-1)}{z-2} \frac{\partial}{\partial z} \right)^a \frac{2-2z}{z-2}.
\end{equation}
The last equality here is obtained from the spectral curve relation $x = \frac{z-1}{z^2}$ as well as the fact that $-2\frac{z-1}{z-2} = \sum_{\mu=1}^\infty \binom{2\mu}{\mu} \, x^\mu$, used in the proof of Proposition~\ref{pro:fenergy02}. It is clear by induction that for all non-negative integers $a$, the function $f_a(z)$ is rational with a pole only at $z = 2$ of order $2a+1$. The table below lists the functions $f_a(z)$ for $0 \leq a \leq 5$.
\begin{center}
\begin{tabular}{ll} \toprule
$a$ & $f_a(z)$ \\ \midrule
0 & $-\frac{2(z-1)}{z-2}$ \\
1 & $- \frac{2z(z-1)}{(z-2)^3}$ \\
2 & $- \frac{2z(z-1) (z^2 + 2z - 2)}{(z-2)^5}$ \\
3 & $- \frac{2z(z-1) (z^4 + 10z^3 - 6z^2 - 8z + 4)}{(z-2)^7}$ \\
4 & $- \frac{2z(z-1) (z^6 + 30z^5 + 42z^4 - 136z^3 + 48z^2 + 24z - 8)}{(z-2)^9}$ \\
5 & $- \frac{2z(z-1) (z^8 + 74z^7 + 442z^6 - 568z^5 - 860z^4 + 1328z^3 - 368z^2 - 64z + 16)}{(z-2)^{11}}$ \\ \bottomrule
\end{tabular}
\end{center}

Using these functions, we obtain the following corollary of the polynomiality of monotone Hurwitz numbers.

\begin{corollary} \label{cor:fenergies}
For $(g,n) \neq (0,1)$ or $(0,2)$, the free energies can be expressed in the following way.
\[
F_{g,n}(z_1, z_2, \ldots, z_n) = \sum_{a_1, a_2, \ldots, a_n = 0}^\mathrm{finite} C_{g,n}(a_1, a_2, \ldots, a_n) \prod_{i=1}^n f_{a_i}(z_i)
\]
In particular, $F_{g,n}(z_1, z_2, \ldots, z_n)$ is a symmetric rational function in $z_1, z_2, \ldots, z_n$ with poles only at $z_i = 2$, for $i = 1, 2, \ldots, n$.
\end{corollary}

One may compare Proposition~\ref{pro:polynomiality} with the classical case of simple Hurwitz numbers, in which the polynomiality result takes the following form, where $P_{g,n}$ is a symmetric polynomial with rational coefficients of degree $3g-3+n$.
\[
H_{g,n}(\mu_1, \mu_2, \ldots, \mu_n) = (2g-2+n+|\mmu|)! \times \prod_{i=1}^n \frac{\mu_i^{\mu_i}}{\mu_i!} \times P_{g,n}(\mu_1, \mu_2, \ldots, \mu_n)
\]
The coefficient of $\mu_1^{a_1} \mu_2^{a_2} \cdots \mu_n^{a_n}$ in the polynomial $P_{g,n}(\mu_1, \mu_2, \ldots, \mu_n)$ is the following intersection number on the moduli space $\overline{\mathcal M}_{g,n}$ of stable pointed curves.\footnote{For definitions of the moduli space $\overline{\mathcal M}_{g,n}$ of stable pointed curves as well as the cohomology classes $\psi_i \in H^2(\overline{\mathcal M}_{g,n}; \mathbb{Q})$ and $\lambda_k \in H^{2k}(\overline{\mathcal M}_{g,n})$, see the book of Harris and Morrison~\cite{har-mor}.}
\[
(-1)^{3g - 3 + n - |\mathbf{a}|} \int_{\overline{\mathcal M}_{g,n}} \psi_1^{a_1} \psi_2^{a_2} \cdots \psi_n^{a_n} \lambda_{3g - 3 + n - |\mathbf{a}|}
\]
This result is precisely the content of the celebrated ELSV formula~\cite{eke-lan-sha-vai}. It would be interesting to similarly determine an ELSV-type formula for monotone Hurwitz numbers that gives an algebro-geometric meaning to the coefficients of $\P_{g,n}$~\cite{gou-gua-nov13b}. The general theory of topological recursion allows us to immediately deduce such a result in the case of the highest degree coefficients. This forms the content of Proposition~\ref{pro:coeffs}.

\section{Topological recursion and quantum curves} \label{sec:toprec}

\subsection{Topological recursion} \label{subsec:toprec}

The topological recursion of Eynard and Orantin was initially inspired by the theory of matrix models~\cite{eyn-ora07a}. It formalises and generalises the loop equations, which may be used to calculate perturbative expansions of matrix model correlation functions. As mentioned in the introduction, the topological recursion uses a spectral curve $\mathcal C$ to define meromorphic multidifferentials $\omega_{g,n}$ on ${\mathcal C}^n$, for integers $g \geq 0$ and $n \geq 1$. More precisely, $\omega_{g,n}$ is a meromorphic section of the line bundle $\pi_1^*(T^*{\mathcal C}) \otimes \pi_2^*(T^*{\mathcal C}) \otimes \cdots \otimes \pi_n^*(T^*{\mathcal C})$ on the Cartesian product ${\mathcal C}^n$, where $\pi_i: {\mathcal C}^n \to {\mathcal C}$ denotes projection onto the $i$th factor.

{\bf Input.} The recursion takes as input a \emph{spectral curve}, which consists of a compact Riemann surface $\mathcal C$ endowed with two meromorphic functions $x$ and $y$, as well as a Torelli marking.\footnote{A\emph{Torelli marking} of $\mathcal C$ is a choice of a symplectic basis of $H_1({\mathcal C}, \mathbb{Z})$.} We require all zeroes of $\mathrm{d}x$ to be simple and distinct from zeroes of $\mathrm{d}y$. Extensions of the topological recursion to more general spectral curves have appeared in the literature, although they are more involved than we need here~\cite{bou-eyn}.

{\bf Base cases.} The base cases for the recursion are
\[
\omega_{0,1}(z_1) = -y(z_1)~\dd x(z_1) \qquad \text{and} \qquad \omega_{0,2}(z_1, z_2) = B(z_1, z_2).
\]
Here, $B(z_1, z_2)$ is the unique meromorphic bidifferential on ${\mathcal C} \times {\mathcal C}$ that
\begin{itemize}
\item is symmetric: $B(z_1, z_2) = B(z_2, z_1)$;
\item is normalised on the $A$-cycles of $H_1({\mathcal C}, \mathbb{Z})$: $\oint_{A_i} B(z_1, \,\cdot\,) = 0$; and
\item has double poles without residue along the diagonal $z_1 = z_2$ but is holomorphic away from the diagonal: $B(z_1, z_2) = \frac{\dd z_1~\dd z_2}{(z_1-z_2)^2} + \text{holomorphic}$.
\end{itemize}
The bidifferential $B(z_1, z_2)$ is a natural construction that is referred to as the fundamental normalised differential of the second kind on $\mathcal C$. In the case that ${\mathcal C} = \mathbb{CP}^1$, the bidifferential is given by the formula $B(z_1, z_2) = \frac{\dd z_1 ~ \dd z_2}{(z_1 - z_2)^2}$.

{\bf Output.} Define the multidifferentials $\omega_{g,n}$ by the following recursive equation, where $S = \{1, 2, \ldots, n\}$ and $z_I = (z_{i_1}, z_{i_2}, \ldots, z_{i_k})$ for $I = \{i_1, i_2, \ldots, i_k\}$.
\begin{equation} \label{eq:toprec}
\omega_{g,n}(z_S) = \sum_{\alpha} \mathop{\text{Res}}_{z=\alpha} \, K(z_1, z) \Bigg[ \omega_{g-1,n+1}(z, \z, z_{S \setminus \{1\}}) + \mathop{\sum_{g_1+g_2=g}^\circ}_{I \sqcup J = S \setminus \{1\}} \omega_{g_1, |I|+1}(z, z_I) \, \omega_{g_2, |J|+1}(\z, z_J) \Bigg]
\end{equation}
Here, the outer summation is over the zeroes $\alpha$ of $\dd x$. Since we have assumed that the zeroes are simple, there exists a unique non-identity meromorphic function $z \mapsto \z$ defined on a neighbourhood of $\alpha \in {\mathcal C}$ such that $x(\z) = x(z)$. The symbol $\circ$ over the inner summation denotes the fact that we exclude all terms that involve $\omega_{0,1}$. The kernel $K$ appearing in the residue is defined by the following equation, where $o$ is an arbitrary base point on the spectral curve.\footnote{Note that the numerator of the kernel here differs from that given in the seminal paper of Eynard and Orantin~\cite{eyn-ora07a}. The modified kernel presented here appears in various subsequent sources in the literature~\cite{bou-eyn}. The two kernels yield precisely the same results, due to general properties of the correlation differentials produced by the topological recursion.}
\[
K(z_1, z) = - \frac{\int_o^z \omega_{0,2}(z_1, \,\cdot\,)}{[y(z) - y(\z)] \, \dd x(z)}
\]

The multidifferentials $\omega_{g,n}$ have been referred to in the literature as both Eynard--Orantin invariants and correlation functions. Since they are neither functions nor invariant under prescribed transformations, we will use the term \emph{correlation differentials}.

The topological recursion has found wide applicability beyond the realm of matrix models, where it was first conceived. It is now known to govern a variety of problems in enumerative geometry and mathematical physics, with conjectural relations to many more. The table below lists some of these alongside their associated spectral curves, with those yet to be rigorously proven below the dashed line.

\begin{center}
\begin{tabular}{@{}p{104.5mm}@{}ll@{}} \toprule
PROBLEM & \multicolumn{2}{@{}l@{}}{SPECTRAL CURVE} \\ \midrule
intersection theory on moduli spaces of curves~\cite{eyn-ora07a} & $x(z) = z^2$ & $y(z) = z$ \\
enumeration of ribbon graphs~\cite{nor13, dum-mul-saf-sor} & $x(z) = z + \frac{1}{z}$ & $y(z) = z$ \\
Weil--Petersson volumes of moduli spaces~\cite{eyn-ora07b} & $x(z) = z^2$ & $y(z) = \frac{\sin(2\pi z)}{2\pi}$ \\ 
stationary Gromov--Witten theory of $\mathbb{P}^1$~\cite{nor-sco, dun-ora-sha-spi} & $x(z) = z + \frac{1}{z}$ & $y(z) = \log(z)$ \\
simple and orbifold Hurwitz numbers~\cite{bor-eyn-mul-saf, eyn-mul-saf, do-lei-nor, bou-her-liu-mul} & $x(z) = z \exp(-z^a)$ & $y(z) = z^a$ \\
Gromov--Witten theory of toric Calabi--Yau threefolds~\cite{eyn-ora12} \hspace{10pt} & \multicolumn{2}{@{}l@{}}{mirror curves} \\
\multicolumn{3}{@{}c@{}}{\makebox[\linewidth]{\dashrule[linkblue]}} \\
enumeration of hypermaps~\cite{do-man} & $x(z) = z^{a-1} + \frac{1}{z}$ & $y(z) = z$ \\
spin Hurwitz numbers~\cite{mul-sha-spi} & $x(z) = z \exp(-z^r)$ & $y(z) = z$ \\
asymptotics of coloured Jones polynomials of knots~\cite{dij-fuj-man, bor-eyn} & \multicolumn{2}{@{}l@{}}{$A$-polynomials} \\ \bottomrule
\end{tabular}
\end{center}

\subsection{Quantum curves} \label{subsec:qcurves}

Following the work of Gukov and Su{\l}kowski, we use the correlation differentials produced by the topological recursion to define \emph{free energies}~\cite{guk-sul}.
\[
F_{g,n}(z_1, z_2, \ldots, z_n) = \int_p^{z_1} \!\! \int_p^{z_2} \! \cdots \! \int_p^{z_n} \omega_{g,n}(z_1, z_2, \ldots, z_n)
\]
We choose a point $p$ on the spectral curve such that $x(p) = \infty$ as the base point for each of the nested integrals~\cite{bor-eyn}. The free energies are in turn used to define the following wave function.
\[
Z(x, \h) = \exp \left[ \sum_{g=0}^\infty \sum_{n=1}^\infty \frac{\h^{2g-2+n}}{n!} \, F_{g,n}(z, z, \ldots, z) \right]
\]
Note that the definition provided here applies only in the case of genus zero spectral curves. For higher genus, it has been proposed in the physics literature that one should include non-perturbative correction terms involving derivatives of theta functions associated to the spectral curve~\cite{bor-eyn}.

For a spectral curve of the form $A(x, y) = 0$, one can ask whether there exists a quantisation $\widehat{A}(\widehat{x}, \widehat{y})$ satisfying certain conditions. First, this quantisation should be non-commutative in the sense that $\widehat{x}$ and $\widehat{y}$ satisfy $[\widehat{x}, \widehat{y}] = \h$. So it is natural for the multiplication operator $\widehat{x} = x$ and the differentiation operator $\widehat{y} = -\h \frac{\partial}{\partial x}$ to be chosen as the polarisation. Second, $\widehat{A}(\widehat{x}, \widehat{y})$ should annihilate the wave function $Z(x, \h)$. So we call $\widehat{A}(\widehat{x}, \widehat{y})$ a \emph{quantum curve} if the spectral curve $A(x, y) = 0$ is recovered in the semi-classical limit $\h \to 0$ and the following Schr\"{o}dinger-like equation is satisfied.
\[
\widehat{A}(\widehat{x}, \widehat{y}) \, Z(x, \h) = 0
\]

Gukov and Su{\l}kowski posit the existence of a quantum curve for a spectral curve $C$ if and only if a certain $K$-theoretic condition is satisfied --- namely, that the tame symbol $\{x, y\} \in K_2(\mathbb{C}(C))$ is a torsion class~\cite{guk-sul}. Note that this condition is automatically satisfied when the spectral curve has genus zero. Moreover, they combine the calculation of the wave function via the topological recursion along with the Schr\"{o}dinger-like equation in order to solve for $\widehat{A}$ order by order in powers of $\h$.
\[
\widehat{A} = \widehat{A}_0 + \h \widehat{A}_1 + \h^2 \widehat{A}_2 + \cdots
\]
The paper of Gukov and Su{\l}kowski demonstrates the efficacy of this quantisation process by calculating the first few terms of $\widehat{A}$ and using these to predict the form of the quantum curve in several cases of geometric interest~\cite{guk-sul}. This approach amounts to performing quantisation by travelling the long way around the following schematic diagram.

\begin{center}
\includegraphics{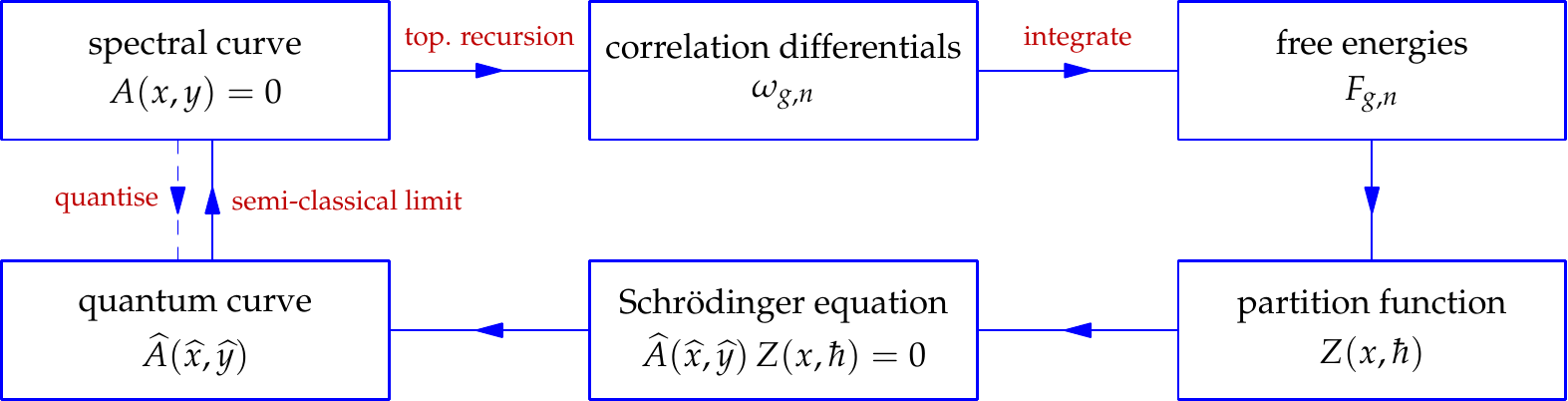}
\end{center}

The quantum curve has been rigorously established for several problems and their associated spectral curves, including intersection theory on moduli spaces of curves~\cite{zho12a}, enumeration of ribbon graphs~\cite{mul-sul}, stationary Gromov--Witten theory of $\mathbb{P}^1$~\cite{dun-mul-nor-pop-sha}, simple Hurwitz numbers and their generalisations~\cite{mul-sha-spi}, and open string invariants for $\mathbb{C}^3$ and the resolved conifold~\cite{zho12b}. The quantum curve for the $A$-polynomial of a knot should recover the $q$-difference operator that appears in the AJ conjecture of Garoufalidis and Le that relates the $A$-polynomial to the coloured Jones polynomials~\cite{gar-le, dij-fuj-man, bor-eyn}.

\section{Topological recursion for monotone Hurwitz numbers} \label{sec:theorem1}

\subsection{Correlation differentials}

In this section, we prove Theorem~\ref{thm:toprec}, which states that the topological recursion applied to the spectral curve $\mathbb{CP}^1$ with functions $x, y: \mathbb{CP}^1 \to \mathbb{C}$ given by
\[
x(z) = \frac{z-1}{z^2} \qquad \text{and} \qquad y(z) = -z
\]
produces correlation differentials whose expansions at $x_i = 0$ satisfy\footnote{Here and for the remainder of the paper, we omit the $\otimes$ in the notation for multidifferentials, so that $\dd x_1 ~ \dd x_2 ~ \cdots \dd x_n$ means $\dd x_1 \otimes \dd x_2 \otimes \cdots \otimes \dd x_n$.}
\[
\omega_{g,n}(z_1, z_2, \ldots, z_n) = \frac{\partial}{\partial x_1} \frac{\partial}{\partial x_2} \cdots \frac{\partial}{\partial x_n} F_{g,n}(z_1, z_2, \ldots, z_n)~\dd x_1 ~ \dd x_2 ~ \cdots ~ \dd x_n, \qquad \text{for $(g,n) \neq (0,2)$}.
\]

For the given spectral curve, the only zero of $\dd x$ occurs at $z = 2$.\footnote{Eynard and Orantin refer to the zeroes of $\dd x$ as \emph{branch points}~\cite{eyn-ora07a}. Note that the meromorphic function $x: \mathbb{CP}^1 \to \mathbb{C}$ here has a branch point in the classical sense at $z = 0$, which is not a zero of $\dd x$. So the point $z = 0$ on the spectral curve does not feature in the topological recursion.} The local involution $z \mapsto \z$ defined near $z = 2$ satisfies $x(\z) = x(z)$, which leads to $\z = \frac{z}{z-1}$.\footnote{The use of $\z$ to denote the local involution should not cause confusion, since complex conjugation plays no role in our work.} In fact, since the meromorphic function $x$ on the spectral curve defines a twofold branched cover of $\mathbb{CP}^1$, this extends to a global involution on the spectral curve that exchanges the two sheets. Finally, recall that the kernel of the topological recursion is given by
\[
K(z_1, z) = - \frac{\int_o^z \omega_{0,2}(z_1, \,\cdot\,)}{[y(z) - y(\z)] \, \dd x(z)} =  \frac{\dd z_1}{z_1-z} \frac{1}{(z-\z) \, \dd x(z)},
\]
where we have chosen the base point $o$ to be the point $z = \infty$ on the spectral curve.

With this information, it is straightforward to compute the correlation differentials $\omega_{g,n}$ for small values of $g$ and $n$. Presently, we require the following calculations, the first two of which are the base cases for the topological recursion.
\begin{align*}
\omega_{0,1}(z_1) &= z_1 ~ \dd x_1 \\
\omega_{0,2}(z_1, z_2) &= \frac{\dd z_1 ~ \dd z_2}{(z_1 - z_2)^2} \\
\omega_{0,3}(z_1, z_2, z_3) &= \frac{8 ~ \dd z_1 ~ \dd z_2 ~ \dd z_3}{(z_1-2)^2 (z_2-2)^2 (z_3-2)^2} \\
\omega_{1,1}(z_1) &= \frac{z_1-1}{(z_1-2)^4} ~ \dd z_1
\end{align*}

Let us define
\begin{equation} \label{eq:omega}
\Omega_{g,n}(z_1, z_2, \ldots, z_n) = \frac{\partial}{\partial x_1} \frac{\partial}{\partial x_2} \cdots \frac{\partial}{\partial x_n} F_{g,n}(z_1, z_2, \ldots, z_n) ~ \dd x_1 ~ \dd x_2 ~ \cdots ~ \dd x_n.
\end{equation}
Corollary~\ref{cor:fenergies} states that for $(g,n) \neq (0,1)$ or $(0,2)$, $\Omega_{g,n}(z_1, z_2, \ldots, z_n)$ is a symmetric rational multidifferential on the $n$-fold Cartesian product of the spectral curve with poles only at $z_i = 2$, for $i = 1, 2, \ldots, n$. Theorem~\ref{thm:toprec} requires us to prove that $\omega_{g,n} = \Omega_{g,n}$ for $(g,n) \neq (0,2)$. For small values of $g$ and $n$, the multidifferentials $\Omega_{g,n}$ can be computed using the generating function form of the monotone cut-and-join recursion of Proposition~\ref{pro:cutjoingf}. This allows us to verify Theorem~\ref{thm:toprec} for small cases.

\begin{proposition} \label{pro:smallcases}
The correlation differentials satisfy $\omega_{g,n} = \Omega_{g,n}$ for $(g,n) = (0,1)$, $(0,3)$, $(1,1)$ and
\[
\omega_{0,2}(z_1, z_2) = \Omega_{0,2}(z_1, z_2) + \frac{\dd x_1 ~ \dd x_2}{(x_1 - x_2)^2}.
\]
\end{proposition}

\begin{proof}
As mentioned above, the correlation differentials $\omega_{g,n}$ can be computed using the topological recursion. On the other hand, the multidifferentials $\Omega_{g,n}$ can be computed using the generating function form of the monotone cut-and-join recursion. We discuss here only the most subtle case, which occurs when $(g,n) = (0,2)$. In the proof of Proposition~\ref{pro:fenergy02}, we obtained the equation
\[
x_1 \frac{\partial}{\partial x_1} F_{0,2}(z_1, z_2) = \frac{z_1 (z_1-1) (z_2-1)}{(z_1-2) (z_1z_2-z_1-z_2)}.
\]
One can divide through by $x_1$ and differentiate with respect to $x_2$ to obtain an expression for $\frac{\Omega_{0,2}(z_1, z_2)}{\dd x_1 ~ \dd x_2}$, which is used in the following calculation.
\begin{align*}
\frac{\Omega_{0,2}(z_1, z_2)}{\dd x_1 ~ \dd x_2} + \frac{1}{(x_1-x_2)^2} &= \frac{z_1^3}{z_1-2} \frac{z_2^3}{z_2-2} \frac{1}{(z_1z_2-z_1-z_2)^2} + \frac{z_1^4 z_2^4}{(z_1-z_2)^2 (z_1z_2-z_1-z_2)^2} \\
&= \frac{z_1^3}{z_1-2} \frac{z_2^3}{z_2-2} \frac{1}{(z_1-z_2)^2} \\
&= \frac{\dd z_1}{\dd x_1} \frac{\dd z_2}{\dd x_2} \frac{1}{(z_1-z_2)^2}
\end{align*}
The desired result is obtained by multiplying both sides of this equation by $\dd x_1 ~ \dd x_2$.
\end{proof}

We require one more preliminary result before commencing the proof of Theorem~\ref{thm:toprec}.

\begin{lemma} \label{lem:sympart}
For $(g,n) \neq (0,1)$ or $(0,2)$, the sum $F_{g,n}(z_1, z_2, \ldots, z_n) + F_{g,n}(\z_1, z_2, \ldots, z_n)$ is holomorphic at $z_1 = 2$. In fact, the sum is independent of $z_1$.
\end{lemma}

\begin{proof}
The proof relies on the functions $f_0(z), f_1(z), f_2(z), \ldots$ defined by equation~(\ref{eq:functions}).
First, use the fact that $f_0(z) + f_0(\z) = f_0(z) + f_0(\tfrac{z}{z-1}) = -2$ to deduce that for all $k \geq 1$,
\[
f_k(z) + f_k(\z) = x \frac{\partial}{\partial x} \left[ f_{k-1}(z) + f_{k-1}(\z) \right] = 0.
\]
Second, use the fact that $f_0(0) = -1$ and the recursive definition of $f_k(z)$ to deduce that for all $k \geq 1$,
\[
f_k(0) = 0.
\]

Now apply these results to the expression for the free energies given in equation~(\ref{eq:fenergies}).
\begin{align*}
F_{g,n}(z_1, z_2, \ldots, z_n) + F_{g,n}(\z_1, z_2, \ldots, z_n) &= \sum_{a_1, a_2, \ldots, a_n = 0}^\text{finite} C_{g,n}(a_1, a_2, \ldots, a_n) \left[ f_{a_1}(z_1) +f_{a_1}(\z_1) \right] \prod_{i=2}^n f_{a_i}(z_i) \\
&= -2 \sum_{a_2, \ldots, a_n = 0}^\text{finite} C_{g,n}(0, a_2, \ldots, a_n) \prod_{i=2}^n f_{a_i}(z_i) \\
&= 2 F_{g,n}(0, z_2, \ldots, z_n)
\end{align*}
Thus, the sum $F_{g,n}(z_1, z_2, \ldots, z_n) + F_{g,n}(\z_1, z_2, \ldots, z_n)$ is independent of $z_1$ and hence, holomorphic at $z_1 = 2$.
\end{proof}

\subsection{Proof of Theorem~\ref{thm:toprec}}

In order to prove Theorem~\ref{thm:toprec}, we adopt the following strategy, which has previously been used to relate enumerative problems to topological recursion on a rational spectral curve. For example, it has proved successful in the case of enumeration of lattice points in moduli spaces of curves~\cite{nor13}, simple Hurwitz numbers~\cite{eyn-mul-saf}, and orbifold Hurwitz numbers~\cite{do-lei-nor,bou-her-liu-mul}.

\begin{itemize}
\item Write the underlying recursion for the enumerative problem --- sometimes known as a Tutte or cut-and-join recursion --- in generating function form. This yields a relation between the free energies which, for the present case, appears as Proposition~\ref{pro:cutjoingf}.

\item Take the symmetric part of the equation obtained, with respect to the local involution at each zero of $\dd x$. Express the equation in terms of multidifferentials rather than functions, before taking the principal part at the zero of $\dd x$. Recall that the \emph{principal part} of a meromorphic 1-form $\Omega(z)$ at $z = \alpha$ may be defined by
\begin{equation} \label{eq:ppart}
[\Omega(z)]_\alpha = \mathop{\mathrm{Res}}_{w=\alpha} \, \frac{\dd z}{z-w} \, \Omega(w).
\end{equation}

\item Add the equations obtained for each zero of $\dd x$ and compare with the topological recursion. Use the fact that in the case of a genus 0 spectral curve, the topological recursion given by equation~(\ref{eq:toprec}) expresses each correlation differential as the sum of its principal parts~\cite[Proposition 16]{do-lei-nor}.
\end{itemize}

The topological recursion for the spectral curve in question is given by the following equation, where $S = \{1, 2, \ldots, n\}$ and $z_I = (z_{i_1}, z_{i_2}, \ldots, z_{i_k})$ for $I = \{i_1, i_2, \ldots, i_k\}$.
\[
\omega_{g,n}(z_S) = \mathop{\text{Res}}_{z=2} \, K(z_1, z) \Bigg[ \omega_{g-1,n+1}(z, \z, z_{S \setminus \{1\}}) + \mathop{\sum_{g_1+g_2=g}^\circ}_{I \sqcup J = S \setminus \{1\}} \omega_{g_1, |I|+1}(z, z_I) \, \omega_{g_2, |J|+1}(\z, z_J) \Bigg]
\]

Proposition~\ref{pro:smallcases} asserts that Theorem~\ref{thm:toprec} is true when $(g,n) = (0,1)$, $(0,3)$ and $(1,1)$. So to establish Theorem~\ref{thm:toprec} in general, it is sufficient to show that the multidifferentials $\Omega_{g,n}$ satisfy the following equation for $(g,n) \neq (0,1)$, $(0,2)$, $(0,3)$ or $(1,1)$.
\begin{equation} \label{eq:topreca}
\begin{aligned}
\Omega_{g,n}(z_S) =& ~ \mathop{\text{Res}}_{z=2} \, K(z_1, z) \Bigg[ \Omega_{g-1,n+1}(z, \z, z_{S \setminus \{1\}}) + \mathop{\sum_{g_1+g_2=g}^{\text{stable}}}_{I \sqcup J = S \setminus \{1\}} \Omega_{g_1, |I|+1}(z, z_I) \, \Omega_{g_2, |J|+1}(\z, z_J) \\
&+ \sum_{i=2}^n \left[ \Omega_{g,n-1}(z, z_{S \setminus \{1, i\}}) \omega_{0,2}(\z, z_i) + \Omega_{g,n-1}(\z, z_{S \setminus \{1, i\}}) \omega_{0,2}(z, z_i) \right] \Bigg].
\end{aligned}
\end{equation}
Here, the word ``stable'' over the summation indicates that we exclude all terms involving $\Omega_{0,1}$ or $\Omega_{0,2}$. Note that we treat the $(0,2)$ terms separately, since $\omega_{0,2} \neq \Omega_{0,2}$.

Our starting point is the monotone cut-and-join recursion, in the generating function form stated as Proposition~\ref{pro:cutjoingf}.
\begin{equation} \label{eq:cut-and-join}
\begin{aligned} 
(1-2x_1 z_1) \, x_1 \frac{\partial}{\partial x_1} F_{g,n}(z_S) &= \sum_{i=2}^n \frac{x_1x_i}{x_1-x_i} \left[ \frac{\partial}{\partial x_1} F_{g,n-1}(z_{S \setminus \{i\}}) - \frac{\partial}{\partial x_i} F_{g,n-1}(z_{S \setminus \{1\}}) \right] \\
&+ \left[ x(t_1) x(t_2) \, \frac{\partial^2}{\partial x(t_1) \, \partial x(t_2)} F_{g-1,n+1}(t_1, t_2, z_{S \setminus \{1\}}) \right]_{t_1 = t_2 = z_1} \\
&+ \mathop{\sum_{g_1 + g_2 = g}}_{I \sqcup J = S \setminus \{1\}}^\circ \left[ x_1 \frac{\partial}{\partial x_1} F_{g_1, |I|+1}(z_1, z_I) \right] \left[ x_1 \frac{\partial}{\partial x_1} F_{g_2, |J|+1}(z_1, z_J) \right]
\end{aligned}
\end{equation}

We apply the involution $z_1 \mapsto \z_1$ to the equation above. More precisely, interpret both sides of the equation above as meromorphic functions on ${\mathcal C}^n$, the $n$-fold Cartesian product of the spectral curve. We pull back both sides using $z_1 \mapsto \z_1$ acting on the first component of ${\mathcal C}^n$.
\begin{equation} \label{eq:cut-and-join-involution}
\begin{aligned} 
(1-2x_1 \z_1) x_1 \frac{\partial}{\partial x_1} F_{g,n}(\z_1, z_{S \setminus \{1\}}) &= \sum_{i=2}^n \frac{x_1x_i}{x_1-x_i} \left[ \frac{\partial}{\partial x_1} F_{g,n-1}(\z_1, z_{S \setminus \{1, i\}}) - \frac{\partial}{\partial x_i} F_{g,n-1}(z_i, z_{S \setminus \{1, i\}}) \right] \\
&+ \left[ x(t_1) x(t_2) \frac{\partial^2}{\partial x(t_1) \, \partial x(t_2)} F_{g-1,n+1}(t_1, t_2, z_{S \setminus \{1\}}) \right]_{t_1 = t_2 = \z_1} \\
&+ \mathop{\sum_{g_1 + g_2 = g}}_{I \sqcup J = S \setminus \{1\}}^\circ \left[ x_1 \frac{\partial}{\partial x_1} F_{g_1, |I|+1}(\z_1, z_I) \right] \left[ x_1 \frac{\partial}{\partial x_1} F_{g_2, |J|+1}(\z_1, z_J) \right].
\end{aligned}
\end{equation}

Now take twice the symmetric part with respect to the involution by adding equations~(\ref{eq:cut-and-join}) and~(\ref{eq:cut-and-join-involution}). The left hand side yields the following, where we have used Lemma~\ref{lem:sympart}.
\begin{align*}
&~ (1-2x_1 z_1) x_1 \frac{\partial}{\partial x_1} F_{g,n}(z_S) + (1-2x_1 \z_1) x_1 \frac{\partial}{\partial x_1} F_{g,n}(\z_1, z_{S \setminus \{1\}}) \\
=&~ (1-2x_1 z_1) x_1 \frac{\partial}{\partial x_1} F_{g,n}(z_S) - (1-2x_1 \z_1) x_1 \frac{\partial}{\partial x_1} F_{g,n}(z_1, z_{S \setminus \{1\}}) \\
=&~ - 2x_1^2 (z_1 - \z_1) \frac{\partial}{\partial x_1} F_{g,n}(z_1, z_{S \setminus \{1\}})
\end{align*}

Similarly, the first term on the right hand side yields the following.
\begin{align*}
\sum_{i=2}^n \frac{-2 x_1 x_i}{x_1-x_i} \frac{\partial}{\partial x_i} F_{g,n-1}(z_{S \setminus \{1\}})
\end{align*}

The second term on the right hand side yields the following, where we have used Lemma~\ref{lem:sympart} to write $F_{g-1,n+1}(z, z, z_{S \setminus \{1\}}) + F_{g-1,n+1}(\z, \z, z_{S \setminus \{1\}}) = -2 F_{g-1,n+1}(z, \z, z_{S \setminus \{1\}}) + G(z_{S \setminus \{1\}})$, for some meromorphic function $G$.
\[
-2 x_1^2 \left[ \frac{\partial^2}{\partial x(t_1) \, \partial x(t_2)} F_{g-1,n+1}(t_1, t_2, z_{S \setminus \{1\}}) \right]_{\begin{subarray}{c} t_1 = z_1 \\ t_2 = \z_1 \end{subarray}}
\]

Similarly, the third term on the right hand side yields the following. Due to the exceptional nature of the $(0,2)$ terms, we consider them separately from the main summation. Again, we use the word ``stable'' over the summation to indicate that we exclude all terms involving $F_{0,1}$ or $F_{0,2}$.
\begin{align*}
-2 x_1^2 \mathop{\sum_{g_1 + g_2 = g}}_{I \sqcup J = S \setminus \{1\}}^{\text{stable}} &\left[ \frac{\partial}{\partial x_1} F_{g_1, |I|+1}(z_1, z_I) \right] \left[ \frac{\partial}{\partial x_1} F_{g_2, |J|+1}(\z_1, z_J) \right] \\
+ 2 \sum_{i=2}^n &\left[ x_1 \frac{\partial}{\partial x_1} F_{0,2}(z_1, z_i) - x_1 \frac{\partial}{\partial x_1} F_{0,2}(\z_1, z_i) \right] \left[ x_1 \frac{\partial}{\partial x_1}F_{g,n-1}(z_{S \setminus \{i\}}) \right]
\end{align*}

Now divide the resulting equation by $-2 x_1^2 (z_1 - \z_1)$ and apply $\frac{\partial}{\partial x_2} \frac{\partial}{\partial x_3} \cdots \frac{\partial}{\partial x_n} \big[ \, \cdot \, \big] ~ \dd x_1 ~ \dd x_2 ~ \cdots ~ \dd x_n$ to both sides. Using the definition of $\Omega_{g,n}$ from equation~(\ref{eq:omega}), we obtain the following.
\begin{align*}
\Omega_{g,n}(z_S) =&~ \frac{1}{(z_1-\z_1) \, \dd x_1} \Bigg[ \sum_{i=2}^n \frac{\partial}{\partial x_i} \left[ \frac{x_i}{x_1(x_1-x_i)} \, \dd x_1 \, \dd x_1 \, \Omega_{g,n-1}(z_{S \setminus \{1\}}) \right] + \Omega_{g-1,n+1}(z_1, \z_1, z_{S \setminus \{1\}}) \\
&+ \mathop{\sum_{g_1 + g_2 = g}}_{I \sqcup J = S \setminus \{1\}}^{\text{stable}} \Omega_{g_1, |I|+1}(z_1, z_I) ~ \Omega_{g_2, |J|+1}(\z_1, z_J) - \sum_{i=2}^n \Omega_{g,n-1}(z_{S \setminus \{i\}}) \left[ \Omega_{0,2}(z_1, z_i) - \Omega_{0,2}(\z_1, z_i) \right] \Bigg]
\end{align*}

Recall the definition of a principal part from equation~(\ref{eq:ppart}) and the fact that a rational meromorphic differential is equal to the sum of its principal parts. In our setting, $\Omega_{g,n}(z_S)$ is a meromorphic differential in $z_1$ with a pole only at $z_1 = 2$ and hence, is equal to its principal part there. Now take the principal part of both sides of the previous equation at $z_1 = 2$ and recall that the recursion kernel is given by $K(z_1, z) = \frac{\dd z_1}{z_1-z} \frac{1}{(z-\z) \, \dd x(z)}$.
\begin{align*}
\Omega_{g,n}(z_S) =&~ \mathop{\mathrm{Res}}_{z=2} \, K(z_1, z) \Bigg[ \sum_{i=2}^n \frac{\partial}{\partial x_i} \left[ \frac{x_i}{x(x-x_i)} \, \dd x \, \dd x \, \Omega_{g,n-1}(z_{S \setminus \{1\}}) \right] + \Omega_{g-1,n+1}(z, \z, z_{S \setminus \{1\}}) \\
&+ \mathop{\sum_{g_1 + g_2 = g}}_{I \sqcup J = S \setminus \{1\}}^{\text{stable}} \Omega_{g_1, |I|+1}(z, z_I) ~ \Omega_{g_2, |J|+1}(\z, z_J) - \sum_{i=2}^n \Omega_{g,n-1}(z_{S \setminus \{i\}}) \left[ \Omega_{0,2}(z, z_i) - \Omega_{0,2}(\z, z_i) \right] \Bigg] \\
=&~ \mathop{\mathrm{Res}}_{z = 2} \, K(z_1, z) \Bigg[ \Omega_{g-1,n+1}(z, \z, z_{S \setminus \{1\}}) + \mathop{\sum_{g_1 + g_2 = g}}_{I \sqcup J = S \setminus \{1\}}^{\text{stable}} \Omega_{g_1, |I|+1}(z, z_I) ~ \Omega_{g_2, |J|+1}(\z, z_J) \\
&- \sum_{i=2}^n \Omega_{g,n-1}(z, z_{S \setminus \{1, i\}}) \left[ \Omega_{0,2}(z, z_i) - \Omega_{0,2}(\z, z_i) \right] \Bigg]
\end{align*}
We have used here the fact that the first term on the right hand side of the first line is equal to zero, since the order two pole of $K(z_1, z)$ at $z = 2$ is removed by the appearance of $\dd x \, \dd x = \frac{(2-z)^2}{z^6} \, \dd z \, \dd z$.

Comparing this with equation~(\ref{eq:topreca}), we see that it suffices to show that for $i = 2, 3, \ldots, n$,
\[
\Omega_{g,n-1}(z, z_{S \setminus \{1, i\}}) \left[ \Omega_{0,2}(\z, z_i) - \Omega_{0,2}(z, z_i) \right] = \Omega_{g,n-1}(z, z_{S \setminus \{1, i\}}) \, \omega_{0,2}(\z, z_i) + \Omega_{g,n-1}(\z, z_{S \setminus \{1, i\}}) \, \omega_{0,2}(z, z_i).
\]
By Lemma~\ref{lem:sympart}, we have $\Omega_{g,n-1}(\z, z_{S \setminus \{1,i\}}) = -\Omega_{g,n-1}(z, z_{S \setminus \{1,i\}})$. The previous equation is then a consequence of the fact that
\[
\Omega_{0,2}(z, z_i) - \omega_{0,2}(z, z_i) = \Omega_{0,2}(\z, z_i) - \omega_{0,2}(\z, z_i).
\]
This is indeed true since both sides are equal to $\frac{\dd x_1 ~ \dd x_2}{(x_1-x_2)^2}$, and that completes the proof of Theorem~\ref{thm:toprec}.

\section{The quantum curve for monotone Hurwitz numbers} \label{sec:theorem2}

\subsection{The wave function}

Whereas Corollary~\ref{cor:fenergies} allows us to interpret the free energy $F_{g,n}(z_1, z_2, \ldots, z_n)$ as a rational function on the spectral curve, we interpret it in this section as the formal power series
\[
F_{g,n}(z_1, z_2, \ldots, z_n) = \sum_{\mu_1, \mu_2, \ldots, \mu_n = 1}^\infty \H_{g,n}(\mu_1, \mu_2, \ldots, \mu_n)~x_1^{\mu_1} x_2^{\mu_2} \cdots x_n^{\mu_n}.
\]
Substituting into the expression for the wave function in equation~(\ref{eq:pfunction}) yields the following.
\[
Z(x, \h) = \exp \left[ \sum_{g=0}^\infty \sum_{n=1}^\infty \frac{\h^{2g-2+n}}{n!} ~ \sum_{\mu_1, \mu_2 \ldots, \mu_n = 1}^\infty \H_{g,n}(\mu_1, \mu_2, \ldots, \mu_n) ~ x^{|\mmu|} \right]
\]
Given that the number $\H_{g,n}(\mmu)$ is non-zero only for $m = 2g-2+n+|\mmu| \geq 0$, it follows that we may interpret the wave function $Z(x, \h)$ as an element of $\mathbb{Q}[[x\h^{-1}, \h]]$.

\begin{proposition} \label{pro:zcoefficients}
Let $f(d, m)$ denote the number of monotone factorisations in $S_d$ with $m$ transpositions. The wave function satisfies
\[
Z(x, \hbar) = 1 + \sum_{d=1}^\infty \sum_{m=0}^\infty \dfrac{f(d, m)}{d!} \, x^d \hbar^{m-d}.
\]
\end{proposition}

\begin{proof}
Let $f^\circ(d,m)$ denote the number of monotone factorisations
\[
\sigma_1 \circ \sigma_2 \circ \cdots \circ \sigma_m = \tau
\]
in $S_d$ with $m$ transpositions, such that $\sigma_1, \sigma_2, \ldots, \sigma_m$ generate a transitive subgroup of $S_d$. The definitions of the monotone Hurwitz numbers and the free energies allow us to write
\[
[x^d]~F_{g,n}(z, z, \ldots, z) = \frac{n!}{d!} \, f^\circ(d,2g-2+n+d).
\]

Therefore, we have the following expression for the wave function.
\begin{align*}
Z(x, \h) &= \exp \left[ \sum_{d=1}^\infty \sum_{m=0}^\infty \dfrac{f^\circ (d, m)}{d!} \, x^{d} \h^{m-d} \right] \\
&= 1 +\sum_{k=1}^\infty \dfrac{1}{k!}\left( \sum_{d=1}^\infty \sum_{m=0}^\infty \dfrac{f^\circ (d, m)}{d!} \, x^{d} \h^{m-d} \right)^k\\
&=1 + \sum_{d=1}^\infty \sum_{m=0}^\infty \left( \sum_{k=1}^\infty \dfrac{1}{k!} \mathop{\sum_{d_1+\cdots+d_k=d}}_{m_1+\cdots+m_k=m} \prod_{i=1}^k \dfrac{f^\circ (d_i,m_i)}{d_i!} \right) x^d \h^{m-d}
\end{align*}

It remains to show that 
\[
f(d,m) = d! \sum_{k=1}^\infty \dfrac{1}{k!} \mathop{\sum_{d_1+\cdots+d_k=d}}_{m_1+\cdots+m_k=m} \prod_{i=1}^k \dfrac{f^\circ (d_i,m_i)}{d_i!}.
\]

For a partition $d_1 + d_2 + \cdots + d_k = d$, there are $\dfrac{d!}{\prod d_i!}$ ways to write $\{1, 2, \ldots, d\}$ as the disjoint union of subsets $X_1, X_2, \ldots, X_k$, where $X_i$ has $d_i$ elements. There are $f^\circ (d_i,m_i)$ transitive monotone factorisations in the symmetric group $S_{X_i}$ with $m_i$ transpositions. Therefore,
\[
d! \, \prod_{i=1}^k \dfrac{f^\circ (d_i,m_i)}{d_i!}
\]
is the number of factorisations in $S_d$ with $m$ transpositions, whose terms can be separated into an ordered tuple of $k$ transitive monotone factorisations in $S_{X_1}, S_{X_2}, \ldots, S_{X_k}$ with $m_1, m_2, \ldots, m_k$ transpositions, respectively. To each such ordered tuple of $k$ factorisations, one recovers a unique monotone factorisation in $S_d$ with $m$ transpositions. Simply arrange the transpositions from all $k$ factorisations in monotone order, while maintaining their relative orders within the original factorisations. There exists a unique way to achieve this, since the transpositions from distinct factorisations commute with each other.

By performing summations over all positive integers $k$, all partitions of $d$ with $k$ parts, and all partitions of $m$ with $k$ parts, we obtain the number of monotone factorisations in $S_d$ with $m$ transpositions, namely $f(d, m)$.
\end{proof}

\begin{lemma}
The number $f(d,m)$ is equal to the Stirling number of the second kind $\stirling{d+m-1}{d-1}$ for $d \geq 1$ and $m \geq 0$.
\end{lemma}

\begin{proof}
We begin by showing that $f(d,m)$ satisfies the recursion
\begin{equation} \label{lem:frec}
f(d,m) = (d-1) \, f(d,m-1) + f(d-1,m).
\end{equation}
By definition, $f(d,m)$ is the number of monotone factorisations in $S_d$ with $m$ transpositions. Since there are $d-1$ transpositions in $S_d$ of the form $(a ~ d)$ with $a < d$, the number of factorisations that contain at least one such term is equal to $(d-1) \, f(d,m-1)$. To deduce the recursion above, one simply notes that there are $f(d-1,m)$ factorisations in $S_d$ with $m$ transpositions that do not contain a term of the form $(a ~ d)$. Note that one can make sense of the recursion in the case $(d,m) = (1,0)$ by defining $f(0,0) = 1$.

Now consider the recursion for the Stirling numbers of the second kind.
\[
\stirling{n+1}{k} = k \, \stirling{n}{k}+\stirling{n}{k-1}
\]
Writing $n = d+m-2$ and $k = d-1$, we obtain the equation
\begin{equation} \label{lem:stirling}
\stirling{d+m-1}{d-1} = (d-1) \, \stirling{d+m-2}{d-1}+\stirling{d+m-2}{d-2}.
\end{equation}
Compare equations~(\ref{lem:frec}) and ~(\ref{lem:stirling}) and check that the base cases $d = 1$ and $m = 0$ are in agreement to conclude by induction that $f(d,m) = \stirling{d+m-1}{d-1}$.
\end{proof}

\subsection{Proof of Theorem~\ref{thm:qcurve}}

In this section, we prove Theorem~\ref{thm:qcurve}, which states that the wave function $Z(x, \hbar)$ satisfies the following equation, where $\widehat{x} = x$ and $\widehat{y} = -\hbar \frac{\partial}{\partial x}$.
\[
[\widehat{x} \widehat{y}^2 + \widehat{y} + 1] \, Z(x, \hbar) = 0
\]

Take equation~(\ref{lem:frec}), multiply by $\dfrac{x^d\h^{m-d}}{(d-1)!}$ and sum over $d \geq 1$ and $m \geq 0$ to obtain
\[
\sum_{d=1}^\infty \sum_{m=0}^\infty d(d-1) \, f(d,m-1) \, \dfrac{x^d \h^{m-d}}{d!} - \sum_{d=1}^\infty \sum_{m=0}^\infty df(d,m) \, \dfrac{x^d \h^{m-d}}{d!} + \sum_{d=1}^\infty \sum_{m=0}^\infty f(d-1,m) \, \dfrac{x^d \h^{m-d}}{(d-1)!} =0.
\] 
We can then adjust the summation variables and use the operator $\frac{\partial}{\partial x}$ to write this as
\[
x^2 \h\dfrac{\partial^2}{\partial x^2}\sum_{d=1}^\infty \sum_{m=0}^\infty \frac{f(d,m)}{d!} \, x^d \h^{m-d} - x\dfrac{\partial}{\partial x} \sum_{d=1}^\infty \sum_{m=0}^\infty \frac{f(d,m)}{d!} \, x^d \h^{m-d} + \frac{x}{\h} \left[ 1 + \sum_{d=1}^\infty \sum_{m=0}^\infty \frac{f(d,m)}{d!} \, x^d \h^{m-d} \right] = 0.
\]
Multiplying through by $\frac{\h}{x}$ and invoking Proposition~\ref{pro:zcoefficients} gives
\[
\left[ x \h^2 \dfrac{\partial^2}{\partial x^2} - \h \dfrac{\partial}{\partial x} + 1 \right] Z(x,\h) = 0.
\]

After making the substitutions $\widehat{x} = x$ and $\widehat{y} = -\hbar \frac{\partial}{\partial x}$, one obtains Theorem~\ref{thm:qcurve}.

\section{Applications} \label{sec:applications}

\subsection{Asymptotic behaviour}

Eynard and Orantin show that applying the topological recursion to the Airy spectral curve
\[
x(z) = z^2 \qquad \text{and} \qquad y(z) = z
\]
produces intersection numbers on the moduli space $\overline{\mathcal M}_{g,n}$ of stable pointed curves~\cite{eyn-ora07a}. More precisely, the correlation differentials in that case satisfy
\[
\omega_{g,n}^{\text{Airy}}(z_1, z_2, \ldots, z_n) = \frac{(-1)^n}{2^{2g-2+n}} \sum_{|\mathbf{a}| = 3g-3+n} \int_{\overline{\mathcal M}_{g,n}} \psi_1^{a_1} \psi_2^{a_2} \cdots \psi_n^{a_n} \, \prod_{i=1}^n \frac{(2a_i+1)!! \, \dd z_i}{z_i^{2a_i+2}}.
\]
They furthermore observe that the topological recursion depends only on the local behaviour of the spectral curve at the zeroes of $\dd x$. Due to the assumption of simple zeroes, this behaviour can in turn be locally modelled on the Airy curve. The upshot is that the asymptotics of the correlation differentials for any spectral curve is closely related to that of the Airy spectral curve. In the case of the monotone Hurwitz numbers, this observation leads to the following result concerning the leading order terms of the polynomial $\P_{g,n}$.

\begin{proposition} \label{pro:coeffs}
If $|\mathbf{a}| = 3g-3+n$, then the coefficient of $\mu_1^{a_1} \mu_2^{a_2} \cdots \mu_n^{a_n}$ in $\P_{g,n}(\mu_1, \mu_2, \ldots, \mu_n)$ is
\[
2^{3g-3+n} \int_{\overline{\mathcal M}_{g,n}} \psi_1^{a_1} \psi_2^{a_2} \cdots \psi_n^{a_n}.
\]
\end{proposition}

Eynard furthermore proposes an interpretation for the lower order behaviour of the correlation differentials in terms of the intersection theory on a moduli space of ``coloured'' stable pointed curves~\cite{eyn11b}. In the present case of a spectral curve with one branch point, the analysis involves only intersection numbers on the usual moduli space $\overline{\mathcal M}_{g,n}$ of stable pointed curves. It would be interesting to determine whether the relation is natural from the algebro-geometric viewpoint, since such a result would constitute a monotone analogue of the ELSV formula~\cite{eke-lan-sha-vai}.

\subsection{String and dilaton equations}

The correlation differentials produced by the topological recursion satisfy \emph{string equations}~\cite{eyn-ora07a}.
\begin{align*}
\sum_\alpha \mathop{\text{Res}}_{z=\alpha} \, y(z) \, \omega_{g,n+1}(z, z_S) &= - \sum_{i=1}^n \dd z_i \, \frac{\partial}{\partial z_i} \left( \frac{\omega_{g,n}(z_S)}{\dd x(z_i)} \right) \\
\sum_\alpha \mathop{\text{Res}}_{z=\alpha} \, x(z) y(z) \, \omega_{g,n+1}(z, z_S) &= - \sum_{i=1}^n \dd z_i \, \frac{\partial}{\partial z_i} \left( \frac{x(z_i) \, \omega_{g,n}(z_S)}{\dd x(z_i)} \right)
\end{align*}
They are also known to satisfy the following \emph{dilaton equation}, where $\dd \Phi(z) = y(z) \, \dd x(z)$.
\begin{equation} \label{eq:dilaton}
\sum_\alpha \mathop{\text{Res}}_{z=\alpha} \, \Phi(z) \, \omega_{g,n+1}(z, z_S) = (2g-2+n) \, \omega_{g,n}(z_S)
\end{equation}

\begin{proposition}[String and dilaton equations for monotone Hurwitz numbers] \label{pro:dilaton}
\begin{align*}
\P_{g,n+1}(-\tfrac{1}{2}, \mmu_S) &= 2|\mmu_S| \, \P_{g,n}(\mmu_S) \\
\P_{g,n+1}(0, \mmu_S) - \P_{g,n+1}(-\tfrac{1}{2}, \mmu_S) &= (2g-2+n) \, \P_{g,n}(\mmu_S)
\end{align*}
\end{proposition}

Observe here that the evaluations of $\P_{g,n+1}$ at 0 and $-\tfrac{1}{2}$ are well-defined, since $\P_{g,n+1}$ is a polynomial. The proofs for the two statements above are similar in nature, so we discuss the dilaton equation only. For this, we require the following residue calculation.

\begin{lemma} \label{lem:residue}
For integers $a \geq 1$ and $d \geq -1$,
\[
\mathop{\mathrm{Res}}_{z=2} \, \frac{f_a(z)}{z^2 (z-2)^d} \, \dd z = ( - \tfrac{1}{2} )^{a+d+1}.
\]
\end{lemma}

\begin{proof}
Let the residue on the left hand side of the equation above be denoted by $R_{a,d}$. We will prove the lemma by induction on $a$. The base case $a = 1$ holds due to the following direct computation.
\[
R_{1,d} = \mathop{\mathrm{Res}}_{z=2} \, \frac{f_1(z)}{z^2 (z-2)^d} \, \dd z
= \mathop{\mathrm{Res}}_{z=2} \, \frac{-2(z-1)}{z (z-2)^{d+3}} \, \dd z
= \mathop{\mathrm{Res}}_{z=2} \, \left[ -2 + \sum_{k=0}^\infty (-\tfrac{1}{2})^k (z-2)^{k-d-3} \right] \dd z \\
= (-\tfrac{1}{2})^{d+2}
\]

One can deduce the lemma for $R_{a+1,d}$ from the statement for $R_{a,d}$, $R_{a,d+1}$ and $R_{a,d+2}$ by the following computation.
\begin{align*}
R_{a+1,d} &= \mathop{\mathrm{Res}}_{z=2} \, \frac{f_{a+1}(z)}{z^2 (z-2)^d} \, \dd z \\
&= \mathop{\mathrm{Res}}_{z=2} \, \dd f_a(z) \, \frac{-z(z-1)}{z^2 (z-2)^{d+1}} \\
&= \mathop{\mathrm{Res}}_{z=2} \, f_a(z) \, \dd \left[ \frac{z(z-1)}{z^2 (z-2)^{d+1}} \right] \\
&= - \mathop{\mathrm{Res}}_{z=2} \, f_a(z) \, \frac{(d+1)(z-2)^2 + (3d+2) (z-2) + (2d+2)}{z^2 (z-2)^{d+2}} \, \dd z \\
&= - (d+1) R_{a,d} - (3d+2) R_{a,d+1} - (2d+2) R_{a,d+2} \\
&= (-\tfrac{1}{2})^{a+d+2}
\end{align*}
The second equality uses the inductive definition of $f_a(z)$ from equation~(\ref{eq:functions}), the third equality uses the fact that $\mathop{\mathrm{Res}} \, \dd (f \, g) = \mathop{\mathrm{Res}} \, (f \, \dd g + g \, \dd f) = 0$, and the final equality uses the inductive hypothesis. Therefore, the desired result follows by induction on $a$.
\end{proof}

\begin{proof}[Proof of Proposition~\ref{pro:dilaton}.]
Consider the left hand side of equation~(\ref{eq:dilaton}) with $\dd \Phi(z) = y(z) \, \dd x(z) = \frac{z-2}{z^2} \, \dd z$.
\begin{align*}
LHS &= \mathop{\mathrm{Res}}_{z=2} \, \Phi(z) \, \omega_{g,n+1}(z, z_S) \\
&= - \mathop{\mathrm{Res}}_{z=2} \, \dd\Phi(z) \int \omega_{g,n+1}(z, z_S) \\
&= - \mathop{\mathrm{Res}}_{z=2} \, \dd\Phi(z) \sum_{a, a_1, a_2, \ldots, a_n = 0}^\mathrm{finite} C_{g,n+1}(a, a_1, a_2, \ldots, a_n) \, f_a(z) \, \prod_{i=1}^n \dd f_{a_i}(z_i)
 \\
&= - \sum_{a, a_1, a_2, \ldots, a_n = 0}^\mathrm{finite} C_{g,n+1}(a, a_1, a_2, \ldots, a_n) \, \prod_{i=1}^n \dd f_{a_i}(z_i) \, \mathop{\mathrm{Res}}_{z=2} \, f_a(z) \frac{z-2}{z^2} \, \dd z \\
&= \sum_{a_1, a_2, \ldots, a_n = 0}^\mathrm{finite} \left[ C_{g,n+1}(0, a_1, a_2, \ldots, a_n) - \sum_{a=0}^{\text{finite}}C_{g,n+1}(a, a_1, a_2, \ldots, a_n) (-\tfrac{1}{2})^a \right] \, \prod_{i=1}^n \dd f_{a_i}(z_i)
\end{align*}
The second equality uses the fact that $\mathop{\mathrm{Res}} \, \dd (f \, g) = \mathop{\mathrm{Res}} \, (f \, \dd g + g \, \dd f) = 0$, while the third equality uses Theorem~\ref{thm:toprec} and Corollary~\ref{cor:fenergies}. The last equality uses Lemma~\ref{lem:residue} in the case $d = -1$ and the direct computation
\[
\mathop{\mathrm{Res}}_{z=2} \, f_0(z) \, \frac{z-2}{z^2} \, \dd z = 0.
\]

Now consider the right hand side of equation~(\ref{eq:dilaton}).
\begin{align*}
RHS &= (2g-2+n) \, \omega_{g,n}(z_S) \\
&= (2g-2+n) \sum_{a_1, a_2, \ldots, a_n = 0}^\mathrm{finite} C_{g,n}(a_1, a_2, \ldots, a_n) \prod_{i=1}^n \dd f_{a_i}(z_i)
\end{align*}

One may equate the coefficients of $\prod \dd f_{a_i}(z_i)$ on both sides of the dilaton equation to obtain the following.
\[
C_{g,n+1}(0, a_1, a_2, \ldots, a_n) - \sum_{a=0}^{\text{finite}} C_{g,n+1}(a, a_1, a_2, \ldots, a_n) (-\tfrac{1}{2})^a = (2g-2+n) \, C_{g,n}(a_1, a_2, \ldots, a_n)
\]
Multiply both sides by $\mu_1^{a_1} \mu_2^{a_2} \cdots \mu_n^{a_n}$ and sum over $\mu_1, \mu_2, \ldots, \mu_n$ to obtain the desired result.
\end{proof}

\bibliographystyle{habbrv}
\bibliography{monotone-hurwitz}

\textsc{School of Mathematical Sciences, Monash University, VIC 3800, Australia} \\
\emph{Email:} \href{mailto:norm.do@monash.edu}{norm.do@monash.edu}

\textsc{School of Mathematical Sciences, Monash University, VIC 3800, Australia} \\
\emph{Email:} \href{mailto:alastair.dyer@gmail.com}{alastair.dyer@gmail.com}

\textsc{School of Mathematical Sciences, Monash University, VIC 3800, Australia} \\
\emph{Email:} \href{mailto:daniel.mathews@monash.edu}{daniel.mathews@monash.edu}

\end{document}